\documentclass[12pt]{amsart}
\usepackage{amsfonts}
\usepackage{graphicx,amssymb,amsfonts,amsmath,amscd}
\usepackage[all,ps,cmtip]{xy}
\usepackage{amscd}

\usepackage{mathrsfs}

\usepackage{eucal}
\usepackage{hyperref}

\newtheorem{defn}{Definition}[section]
\newtheorem{thm}[defn]{Theorem}
\newtheorem{lem}[defn]{Lemma}
\newtheorem{ques}[defn]{Question}
\newtheorem{statement}[defn]{Statement}
\newtheorem*{prob}{Open problem}
\newtheorem{cor}[defn]{Corollary}
\newtheorem{prop}[defn]{Proposition}
\newtheorem{rem}[defn]{Remark}

\newcommand{\N}{\mathbb{N}}

\newcommand{\C}{\mathbb{C}}

\newcommand{\T}{\mathbb{T}}

\newcommand{\om}{\otimes_{min}}

\newcommand{\ncb}{\|_{cb}}
\input xy
\xyoption{all}

\begin{document}
\title{On the $\text{OUMD}$ property for the column Hilbert space $C$}
\author{Yanqi QIU}
\address{Equipe d'Analyse, Universit\'e Paris VI, Paris Cedex 05, France}
\email{yanqi-qiu@math.jussieu.fr}
\keywords{column Hilbert space, operator space $\text{OUMD}$ property, noncommutative $L_p$ spaces}
\begin{abstract}
The operator space $\text{OUMD}$ property was introduced by Pisier in the context of vector-valued noncommutative $L_p$-spaces. It is an open problem whether the column Hilbert space has this property. Based on some complex interpolation techniques, we are able to relate this problem to the study of the $\text{OUMD}_q$ property for non-commutative $L_p$-spaces.
\end{abstract}
\maketitle

\section{Introduction}
In Banach space valued martingale theory, the $\text{UMD}$ property plays an important role. Let us recall briefly the definition of the $\text{UMD}$ property. Let $1<p<\infty$. A Banach space $B$ is $\text{UMD}_p$, if there exists a positive constant which depends only on $p$ and the Banach space $B$ (the best one is usually denoted by $\beta_p(B)$), such that for all positive integers $n$, all sequences $\varepsilon = (\varepsilon_k)_{k=1}^n$ of numbers in $\{ -1, 1 \}$ and all $B$-valued martingale difference sequences $dx = (dx_k)_{k=1}^n$, we have
\begin{displaymath}
\| \sum_{k=1}^n \varepsilon_k dx_k \|_{L_p(B)} \leq \beta_p(B) \|\sum_{k=1}^n dx_k \|_{L_p(B)}.
\end{displaymath}
The $\text{UMD}$ property has very deep connections with the boundedness of certain singular integral operators such as the Hilbert transform, see e.g. Burkholder's article \cite{Burkholder3}. Burkholder and McConnell \cite{Burkholder1} proved that if a Banach space $B$ is $\text{UMD}_p$, then the Hilbert transform is bounded on the Bochner space $L_p(\T,m; B)$. Bourgain \cite{Bourgain} showed that if the Hilbert transform is bounded on $L_p(\T, m; B)$, then $B$ is $\text{UMD}_p$. The fact that the $\text{UMD}$ property is independent of $p$ was first proved by Pisier (using the Burkholder-Gundy extrapolation techniques). More precisely, he proved that the finiteness of $\beta_p(B)$ for some $ 1 < p < \infty$ implies its finiteness for all $ 1 < p < \infty $. Examples of $\text{UMD}$ spaces include all the finite dimensional Banach spaces, the Schatten $p$-classes $S_p$ and more generally the noncommutative $L_p$-spaces associated to a von Neumann algebra $\CMcal{M}$, for all $1<p<\infty$. The readers are referred to Burkholder \cite{Burkholder2, Burkholder3} for information on $\text{UMD}$ spaces.

In his monograph \cite{Pisier2}, Pisier developed a theory of vector-valued noncommutative $L_p$-spaces. For a given hyperfinite von Neumann algebra $\CMcal{M}$ equipped with an normal, semifinite, faithful trace $\tau$ and a given operator space $E$, he defined the space $L_p(\tau;E)$. In the case where $\CMcal{M} = B(\ell_2)$ equipped with the usual trace, this space is denoted by $S_p[E]$. The readers are referred to Pisier's book \cite{Pisier1} for the details on operator space theory. Noncommutative conditional expectations and martingales arise naturally in this setting. Following Pisier, we say that an operator space $E$ is $\text{OUMD}_p$ for some $1 < p < \infty$, if there exists a constant (as before, the best one is usually denoted by $\beta_p^{os}(E)$, which depends on $p$ and the operator space structure on $E$) such that for any interger $n \geq 1$, any $\varepsilon = (\varepsilon_k)_{k=1}^n \in \{ -1, 1 \}^n$, and any $E$-valued noncommutative martingale $(f_k)_{k=0}^n$ in $L_p(\tau;E)$ associated to any increasing filtration $(\CMcal{M}_k)_{k=1}^n$, we have
\begin{displaymath}
\| f_0 + \sum_{k=1}^n \varepsilon_k (f_k - f_{k-1}) \|_{L_p(\tau; E)} \leq \beta_p^{os}(E) \|f_n\|_{L_p(\tau;E)}.
\end{displaymath}

The first remarkable result concerning the $\text{OUMD}$ property was proved in \cite{Pisier_Xu} and \cite{Pisier_Xu_en} by Pisier and Xu, where they proved that all the noncommutative $L_p$-spaces are $\text{OUMD}_p$. In particular, the Schatten $p$-class $S_p$ is $\text{OUMD}_p$. By a well known (but still unpublished) result of Musat, an operator space $E$ is $\text{OUMD}_p$ if and only if $S_p[E]$ is $\text{UMD}$.

The following question is well-known to the experts, it is due to Z.-J. Ruan. 
\begin{ques}\label{ruan}
Does the column Hilbert space $C$ have $\text{OUMD}_p$ property for some or all $ 1 < p < \infty$?
\end{ques}

The main theorem of this paper is:

\begin{thm}\label{mainthm}
Let $1 < p< \infty$, then there exist $1< u, v< \infty$, such that we have an isometric (Banach) embedding $$ S_p[C] \hookrightarrow S_u[S_v].$$
\end{thm}

In principle, this embedding theorem should solve the Ruan problem, since in  \cite{Musat1}, Musat presented a proof of the following:

\begin{statement}\label{Sq}
For any $1< p, q < \infty$, the operator space $S_q$ is $\text{OUMD}_p$. In particular, as a Banach space, $S_p[S_q]$ is $\text{UMD}$.\end{statement}

However, very recently, Javier Parcet found a gap in the proof of the above result. The main gap is in the proof of Proposition 4.10 in \cite{Musat1}, where the system (4.19) and (4.20) has solution only for $\frac{3}{2} < p < 3$. For the reader's convenience, we reproduce in the appendix the results in \cite{Musat1} which are not affected by this gap. 

Since Statement \ref{Sq} is in doubt, the answer to Ruan's question might be negative. Actually, this would be more interesting, since it would then imply that the $\text{OUMD}_p$ property depends on $p$. Whether the $\text{OUMD}_p$ property depends on $p$ or not is one of the main open problems in this direction.

We also prove the following non-embeddability result, showing that the obtained embedding result in Theorem \ref{mainthm} can not be extended to the category of operator spaces.

\begin{thm}\label{bisthm}
Let $1< p_1, p_2, \cdots, p_n < \infty$. The operator space $C$ can not be embedded completely isomorphically into any quotient of $S_{p_1}[S_{p_2}[\cdots [S_{p_n}]\cdots ]]$.
\end{thm}

\section{Preliminaries}\label{pre}   
By an operator space we mean a closed subspace of $B(H)$ for some complex Hilbert space $H$. When $E \subset B(H)$ is an operator space, we denote by $M_n(E)$ the space of all $n\times n$ matrices with entries in $E$, equipped with the norm induced by $B(\ell_2^n\otimes_2H)$. Let $e_{ij}$ be the element of $B(\ell_2)$ corresponding to the matrix whose coefficients equal to one at the $(i,j)$ entry and zero elsewhere. The column Hilbert space $C$ is defined as
\begin{displaymath}
C = \overline{\textrm{span}}\{ e_{i1}|i \geq 1\}
\end{displaymath} 
and the row Hilbert space $R$ is defined as
\begin{displaymath}
R = \overline{\textrm{span}}\{ e_{1j} | j \geq 1 \}.
\end{displaymath}
Their operator space structures are given by the embeddings $C \subset B(\ell_2)$ and $R \subset B(\ell_2)$.

Z.-J. Ruan in \cite{Ruan} gave an abstract characterization of operator spaces in terms of matrix norms. An abstract operator space is a vector space $E$ equipped with matrix norms $\| \cdot \|_m$ on $M_m(E)$ for each positive integer $m$, satisfying the axioms: for all $x \in M_m(E), y \in M_n(E)$ and $\alpha, \beta \in M_m(\C)$, we have
\begin{displaymath}
\left\|\left( \begin{array}{cc} x & 0 \\ 0 & y\end{array}\right)\right\|_{m+n} = \max\{\|x\|_m, \|y\|_n\}, \qquad \|\alpha x \beta \|_m \leq \|\alpha\|\|x\|_m\|\beta\|. 
\end{displaymath}
This abstract characterization can be used to define various important constructions of new operator spaces from given ones. Among these are the projective tensor product, the quotient, the dual etc. The two constructions used frequently in this paper are the Haagerup tensor product and the complex interpolation for operator spaces. Let us recall briefly their definitions and main properties. 

Let $E, F$ be two operator spaces, the Haagerup tensor product $E\otimes_h F$ of $E$ and $F$ is defined as the completion of $E\otimes F$ with respect to the matrix norms
\begin{displaymath}
\|u\|_{h,m} = \inf \{ \|v \| \| w \|: u = v\odot w, v \in M_{m, r}(E), w \in M_{r,m}(F), r\in\N\},  
\end{displaymath}
where the element $v\odot w \in M_m(E\otimes F)$ is defined by $(v\odot w)_{ij} = \sum_{k=1}^m v_{ik}\otimes w_{kj}$, for all $1 \leq i, j \leq m$. 

We refer to \cite{Interpolation} for details about interpolations of Banach spaces. Now let $E_0, E_1$ be two operator spaces, such that $(E_0, E_1)$ is a compatible couple in the sense of \cite{Interpolation}. Following Pisier, we endow the interpolation space $E_\theta = (E_0, E_1)_\theta$ with a canonical operator space structure by defining for all positive integers $m$,
\begin{displaymath}
M_m(E_\theta) = (M_m(E_0), M_m(E_1))_\theta.
\end{displaymath} 

The Haagerup tensor product is injective, projective, self-dual in the finite dimensional case, however, it is not commutative, that is we do not have $E\otimes_h F = F\otimes_h E$ in general. 

We state the following Kouba's interpolation theorem, which mainly says that the Haagerup tensor product behaves nicely with respect to the complex interpolation, for its proof, see e.g. \cite{Pisier5}.

\begin{thm}[Kouba]\label{Kouba}
Let $(E_0, E_1)$ and $(F_0, F_1)$ be two compatible couples of operator spaces. Then $(E_0 \otimes_h F_0, E_1\otimes_h F_1)$ is a compatible couple. Moreover, for all $0 < \theta < 1$ we have complete isometry
\begin{displaymath}
(E_0\otimes_h F_0, E_1\otimes_h F_1)_\theta = (E_0, E_1)_\theta \otimes_h (F_0, F_1)_\theta.
\end{displaymath}
\end{thm}

Let us now turn to some basic definitions of vector-valued noncommutative $L_p$-spaces. Let $S_\infty$ be the space of compact operators on $\ell_2$. It is viewed as an operator space by the natural embedding $S_\infty \subset B(\ell_2)$. The finite dimensional version is $S_\infty^n = B(\ell_2^n)$. It is well known that the trace class $S_1$ (resp. $S_1^n$) is the dual space of $S_\infty$ (resp. $S_\infty^n$). By this duality, we equip $S_1$ (resp. $S_1^n$) with the dual operator space structure. 

Let $E$ be an operator space. Following Pisier, the noncommutative vector-valued $L_p$-spaces in the discrete case are defined by
\begin{displaymath}
S_\infty^n[E] = S_\infty^n\om E, \qquad S_\infty[E] = S_\infty \om E,
\end{displaymath}
\begin{displaymath}
S_1^n[E] = S_1^n \otimes^\wedge E, \qquad S_1[E] = S_1 \otimes^\wedge E.
\end{displaymath}
It turns out that $(S_\infty[E], S_1[E])$ (resp. $(S_\infty^n[E], S_1^n[E])$) is a compatible couple. For $1 < p < \infty$, the definitions are
\begin{displaymath}
S_p^n[E] = (S_\infty^n[E], S_1^n[E])_{\frac{1}{p}}, \qquad S_p[E] = (S_\infty[E], S_1[E])_{\frac{1}{p}}.
\end{displaymath}

Let $C_p$ and $R_p$ denote respectively the column and the row subspace of $S_p$, for $1 \leq p \leq \infty$. In particular, $C_\infty = C$ and $R_\infty = R$ are the column and row Hilbert space defined as above. By well known results, we have
\begin{displaymath}
S_p = (S_\infty, S_1)_{\frac{1}{p}}.
\end{displaymath}
By this identity, $S_p$ is equipped with an operator space structure, which is called the canonical operator space structure on $S_p$. We endow $C_p$ and $R_p$ with the induced operator space structure by inclusions $C_p \subset S_p$ and $R_p \subset S_p$. Let $p'$ denote the conjugate exponent of $p$, i.e. $\frac{1}{p} + \frac{1}{p'} = 1$. Then the natural identification is a complete isometry
\begin{displaymath}
C_p \simeq R_{p'}.
\end{displaymath}
This identification will be used freely in the sequel. Note also we have complete isometry
\begin{displaymath}
C_p^* \simeq C_{p'} \simeq R_p.
\end{displaymath}
The following complete isometries from \cite{Pisier2} will also be used:
\begin{displaymath}
C_p = (C_\infty, C_1)_{\frac{1}{p}}.
\end{displaymath}
More generally, if $\frac{1}{p_\theta} = \frac{1-\theta}{p_0} + \frac{\theta}{p_1}$, then 
\begin{displaymath}
C_{p_\theta} = (C_{p_0}, C_{p_1})_\theta.
\end{displaymath}
With these notations, we have complete isometry
\begin{displaymath}
S_p[E] = C_p\otimes_h E \otimes_h R_p.
\end{displaymath}

The following proposition from \cite{Pisier2} is useful for us.

\begin{prop}\label{description_S_p_E}
Let $E$ be an operator space. Then we have the following complete isometric isomorphism $$ S_p[E] = C_p \otimes_h E \otimes_h R_p. $$
\end{prop}

By a well known (but still unpublished) result of Musat, an operator space $E$ is $\text{OUMD}_p$ if and only if $S_p[E]$ is $\text{UMD}$, hence Question \ref{ruan} is equivalent to the following
\begin{ques}
Let $1< p< \infty$. Does the Banach space $S_p[C]$ have $\text{UMD}$ property?
\end{ques}

\section{The Banach space $S_p[C]$ }\label{main}

We now turn to the proof of Theorem \ref{mainthm}. Let us begin by a simple proposition. 
\begin{prop}
Let  $1 \le u, v \le \infty$. Then $C_u \otimes_h C_v$ is isometric to a Schatten $p$-class $S_p$ for certain $1\le p \le \infty$. In particular, either $1 \le u, v < \infty$ or $1 < u, v \le \infty$, the space $C_u \otimes_h C_v$ is a $\text{UMD}$ Banach space.
\end{prop}
\begin{proof}
Define $\theta = \frac{1}{v}, \eta = \frac{1}{u} \in [0, 1]$,  then we have $$\frac{1}{v} = \frac{1-\theta}{\infty} + \frac{\theta}{1},$$ $$\frac{1}{u} =\frac{1-\eta}{\infty} + \frac{\eta}{1} .$$ Then by Kouba's interpolation result,  $$C_u \otimes_h C_\infty = (C_\infty \otimes_h C_\infty, C_1 \otimes_h C_\infty)_\eta.$$ By applying the isometric identities $$C_\infty \otimes_h C_\infty = S_2, C_1 \otimes_h C_\infty = S_1,$$ we have $$C_u \otimes_h C_\infty = (S_2, S_1)_\eta = S_{\frac{2}{1+ \eta}}.$$ Similarly, we have isometric identity $$C_u \otimes_h C_1 = S_{\frac{2}{\eta}}.$$ Finally, we obtain that 
$$C_u \otimes_h C_v = (C_u \otimes_h C_\infty, C_u \otimes_h C_1)_\theta = (S_{\frac{2}{1+ \eta}}, S_{\frac{2}{\eta}})_\theta = S_{\frac{2}{1-1/v + 1/u}}.$$ 
The second assertion now follows easily.
\end{proof}

The following simple observation will be useful for us.

\begin{rem}\label{rem}
We have complete isometries
\begin{displaymath}
C \otimes_h C = C \om C \simeq C, \quad R \otimes_h R = R \om R \simeq R. 
\end{displaymath}
An application of Kouba's interpolation result yields complete isometry
\begin{equation}\label{iso}
C_p \otimes_h C_p \simeq C_p
\end{equation}
for all $1\leq p \leq \infty$.

More generally, for any integer $n \geq 1$ we have the following complete isometry
\begin{displaymath}
\underbrace{C_p \otimes_h C_p \otimes_h \cdots \otimes_h C_p}_{n \, \, \textit{times}} \simeq C_p.
\end{displaymath}
In particular, we have the following isometry (in the Banach space category)
\begin{displaymath}
\underbrace{C_1 \otimes_h C_1 \otimes_h \cdots \otimes_h C_1}_{n \, \, \textit{times}}\simeq \underbrace{C_\infty \otimes_h C_\infty \otimes_h \cdots \otimes_h C_\infty}_{n \, \, \textit{times}} \simeq \ell_2.
\end{displaymath}
\end{rem}

\begin{proof}[Proof of Theorem \ref{mainthm}]
Let us first assume that $1 < p \le 2$. Define $\theta = \frac{1 + 1/p}{2} \in (0, 1),$ then $ q = \theta p  = \frac{p+1}{2} \in (1, \infty)$ and $ r = \theta p'  = \frac{p+1}{2(p-1)} \in (1, \infty)$. That is $$ \frac{1}{p} = \frac{1-\theta}{\infty}  + \frac{\theta}{q},$$ $$ \frac{1}{p'} = \frac{1-\theta}{\infty}+ \frac{\theta}{r}.$$ By Proposition \ref{description_S_p_E} and Kouba's interpolation result, \begin{eqnarray*} S_p[C] &= &C_p \otimes_h C_\infty \otimes_h C_{p'} \\ &= &(C_\infty \otimes_h C_\infty \otimes_h C_\infty, C_q\otimes_h C_\infty\otimes_h C_r)_\theta \\ \text{By Remark \ref{rem}} &\stackrel{isometric}{=}& (C_1 \otimes_h C_1 \otimes_h C_1, C_q\otimes_h C_\infty\otimes_h C_r)_\theta  \\ & = & C_{\frac{2p}{p+1}} \otimes_h C_{\frac{2p}{p-1} }\otimes_h C_{\frac{2p}{3p-3}} \\ & = & C_{\frac{2p}{p+1}} \otimes_h R_{\frac{2p}{p+1} }\otimes_h R_{\frac{2p}{3-p}} \\& = & S_{\frac{2p}{p+1}} \otimes_h R_{\frac{2p}{3-p}}. \end{eqnarray*} Hence we get the desired isometric embedding  $$S_p[C] \hookrightarrow C_{\frac{2p}{3-p}} \otimes_h S_{\frac{2p}{p+1}} \otimes_h R_{\frac{2p}{3-p}} = S_{\frac{2p}{3-p}} [S_{\frac{2p}{p+1}}].$$ Similar argument shows that if $1 < p \le 2$, then isometrically, we have $$S_{p'}[R] = R_p \otimes_h R_\infty \otimes_h R_{p'} =  R_{\frac{2p}{p+1}} \otimes_h R_{\frac{2p}{p-1} }\otimes_h R_{\frac{2p}{3p-3}} = S_{\frac{2p}{p-1}} \otimes_h R_{\frac{2p}{3p-3}} \hookrightarrow S_{\frac{2p}{3p-3}} [S_{\frac{2p}{p-1}}].$$  By taking the opposite operator space, we have  $$(S_{p'} [R])^{op}  = (R_p \otimes_h R_\infty \otimes_h R_{p'})^{op} = R_{p'}^{op} \otimes_h R_\infty^{op} \otimes_h R_{p}^{op} = C_{p'} \otimes_h C_\infty \otimes_h C_p = S_{p'}[C].$$ It follows that the Banach space $S_{p'}[C]$ embeds isometrically in $S_{\frac{2p}{3p-3}} [S_{\frac{2p}{p-1}}]$ whenever $ 1 < p \le 2$. In other words, if $2 \le p < \infty$, then  $$S_p[C] \hookrightarrow S_{2p/3}[S_{2p}]. $$
\end{proof}

\begin{cor}
If Statement \ref{Sq} holds, then $S_p[C]$ is a $\text{UMD}$ space.
\end{cor}

We will use the following lemma to prove Theorem \ref{bisthm}.

\begin{lem}\label{superR}
If $1\le p_1, p_2, \cdots, p_n < \infty$ or $1< p_1, p_2, \cdots, p_n \le \infty$. Then there exist $1 < q_1, q_2, \cdots, q_n < \infty$, such that we have $$C_{p_1} \otimes_h C_{p_2} \otimes_h \cdots \otimes_h C_{p_n} \stackrel{isometric}{=} C_{q_1} \otimes_h C_{q_2} \otimes_h \cdots \otimes_h C_{q_n} .$$ Moreover, $C_{p_1} \otimes_h C_{p_2} \otimes_h \cdots \otimes_h C_{p_n}$ is a super-reflexive Banach space.
\end{lem}
\begin{proof}
Assume first that $ 1< p_1, p_2, \cdots, p_n \le \infty$, then by choosing $\theta \in (0, 1)$ such that $\theta > \max( 1/p_1, 1/p_2, \cdots, 1/p_n)$, we can define $\widetilde{p}_1, \widetilde{p}_2, \cdots, \widetilde{p}_n \in (1, \infty]$ by $\widetilde{p}_1 = \theta p_1, \widetilde{p}_2 =\theta p_2, \cdots, \widetilde{p}_n = \theta p_n$ such that for $k = 1, 2, \cdots, n$, $$\frac{1}{p_k} = \frac{1-\theta}{\infty} + \frac{\theta}{\widetilde{p}_k}.$$ It follows that \begin{eqnarray*} & & C_{p_1} \otimes_h C_{p_2} \otimes_h \cdots \otimes_h C_{p_n} \\ &=& (C_\infty \otimes_h C_\infty \otimes_h \cdots \otimes_h C_\infty, C_{\widetilde{p}_1} \otimes_h C_{\widetilde{p}_2} \otimes_h \cdots \otimes_h C_{\widetilde{p}_n})_\theta \\ & \stackrel{isometric}{=}&(C_1 \otimes_h C_1 \otimes_h \cdots \otimes_h C_1, C_{\widetilde{p}_1} \otimes_h C_{\widetilde{p}_2} \otimes_h \cdots \otimes_h C_{\widetilde{p}_n})_\theta \\ & = & C_{q_1} \otimes_h C_{q_2} \otimes_h \cdots \otimes_h C_{q_n},\end{eqnarray*} where $\frac{1}{q_k} = \frac{1-\theta}{1} + \frac{\theta}{\widetilde{p}_k} \in (0, 1)$, and hence $1< q_k< \infty$ for all $1\le k \le n$. 

Super-reflexivity of $C_{p_1} \otimes_h C_{p_2} \otimes_h \cdots \otimes_h C_{p_n}$ follows from the above first identity, which shows that it is a $(1-\theta)$-Hilbertian space. 

The case when $1< p_1, p_2, \cdots, p_n \le \infty$ can be treated similarly or can be obtained by duality.
\end{proof}

\begin{proof}[Proof of Theorem \ref{bisthm}]
Assume that we have a complete isomorphic embedding $$j: C \rightarrow S_{p_1}[S_{p_2}[\cdots [S_{p_n}]\cdots ]].$$ By the injectivity of the Haagerup tensor product, we have a complete isomorphic embedding
$$j \otimes Id_R: C \otimes_h R \rightarrow S_{p_1}[S_{p_2}[\cdots [S_{p_n}]\cdots ]] \otimes_h R.$$ Since $1< p_1, p_2, \cdots, p_n < \infty$ and $R = C_1$, hence by Lemma \ref{superR}, $S_{p_1}[S_{p_2}[\cdots [S_{p_n}]\cdots ]] \otimes_h R$ is a super-reflexive Banach space. This implies that $S_\infty = C \otimes_h R$ is also super-reflexive, which is a contradiction.

For any closed subspace $F \subset S_{p_1}[S_{p_2}[\cdots [S_{p_n}]\cdots ]]$, we have 
\begin{displaymath}
\frac{S_{p_1}[S_{p_2}[\cdots [S_{p_n}]\cdots ]]}{F} \otimes_h R \simeq \frac{S_{p_1}[S_{p_2}[\cdots [S_{p_n}]\cdots ]] \otimes_h R}{F\otimes_h R}.
\end{displaymath}
Indeed, by \cite{Pisier2}, for any $1\leq p \leq \infty$, if $E_2\subset E_1$ is a closed subspace, then we have complete isometry
\begin{displaymath}
S_p[E_1/E_2] = S_p[E_1]/S_p[E_2].
\end{displaymath}
Using the above fact, it is easy to see that 
\begin{displaymath}
(E_1/E_2 )\otimes_h R_p = \frac{E_1\otimes_h R_p}{E_2 \otimes_h R_p}.
\end{displaymath}
Note that the super-reflexive property is stable under taking the quotient, hence $\frac{S_{p_1}[S_{p_2}[\cdots [S_{p_n}]\cdots ]]}{F} \otimes_h R$ is super-reflexive. Hence by using the same idea as above, assume that there is a completely isomorphic embedding: $$i: C \rightarrow S_{p_1}[S_{p_2}[\cdots [S_{p_n}]\cdots ]]/F,$$ then we have completely isomorphic embedding: $$ i \otimes Id_R: C \otimes_h R \rightarrow \frac{S_{p_1}[S_{p_2}[\cdots [S_{p_n}]\cdots ]]}{F} \otimes_h R,$$ which leads to a contradiction. Hence $C$ can not be embedded completely isomorphically into $S_{p_1}[S_{p_2}[\cdots [S_{p_n}]\cdots ]]/F$.
\end{proof}

\begin{prob} Let $n \ge 3$. Consider the underlying Banach space structure of the operator space $C_{p_1} \otimes_h C_{p_2} \otimes_h \cdots \otimes_h C_{p_n}$. Is it always a $\text{UMD}$ space whenever $1< p_1, p_2, \cdots, p_n< \infty$?
\end{prob}

Let us state the following result of the author from \cite{Q2}.

Fix $1 \le p, q \le \infty$. We define by induction: $$E_0 = \C \text{ \,and\, } E_{n+1}=\ell_p(\ell_q(E_n)).$$ 

\begin{thm}
Let $1\le p \neq q \le \infty$. Then there exists $c=c(p,q) > 1$ depending only on $p$ and $q$, such that the $\text{UMD}_2$ constants of the above defined spaces $E_n$ satisfy $$\beta_2(E_n) \ge c^n.$$ 
\end{thm}

Although we can not solve the above open problem, we have the following:

\begin{prop}
Let $1< p \ne q < \infty$. Define $$X_n(p,q) = \underbrace{C_p \otimes_h C_q \otimes_h \cdots \otimes_h C_p \otimes_h C_q}_{\text{$n$ times $C_p \otimes_h C_q$}}.$$  Then we have $$\lim_{n \to \infty} \beta_2(X_n(p,q)) = \infty.$$
\end{prop}

\begin{proof}
To simplify, if there is no risk of confusion, we will use the notation $X_n$ instead of $X_n(p,q)$ in the proof. Let us first assume that $1< p\ne q < \infty$ and $\frac{1}{p} + \frac{1}{q} = 1$ (i.e. $p = q'$ and $p \ne 2$). Then $C_p = R_q$ and $C_q = R_p$. Hence we have \begin{eqnarray*} X_{2n} &= & C_p \otimes_h C_q \otimes_h\underbrace{C_p \otimes_h C_q \otimes_h \cdots \otimes_h C_p \otimes_h C_q}_{\text{$2(n-1)$ times $C_p \otimes_h C_q$}} \otimes_h C_p \otimes_h C_q \\ & = & C_p \otimes_h C_q \otimes_h X_{2(n-1)} \otimes_h R_q \otimes_h R_p \\ & = & S_p[S_q[X_{2(n-1)}]].\end{eqnarray*} It follows that $$X_{2n} = \underbrace{S_p[S_q[\cdots [S_p[S_q]]\cdots]]}_{\text{$n$ times $S_p[S_q]$}}.$$ In particular, $E_n\subset X_{2n}$ (completely) isometrically. Hence $$\lim_{n \to \infty} \beta_2(X_n) \ge \lim_{n \to \infty} \beta_2(E_n) = \infty.$$

Then we treat the general case. It is easy to see (for example, one can draw the picture of the points $(\frac{1}{p}, \frac{1}{q}), (\frac{1}{r}, \frac{1}{r'}), (\frac{1}{s}, \frac{1}{s})$ in the unit square $(0,1) \times (0,1)$ and use the obvious geometric meaning of the following equation system) that for every pair $(p,q)$ with $1< p \ne q < \infty$, there exist $0 < \theta < 1$ and $1< r, s < \infty$, $r \ne 2$ such that $$ \frac{1}{r} = \frac{1-\theta}{s} + \frac{\theta}{p}$$ $$ \frac{1}{r'} = \frac{1-\theta}{s} + \frac{\theta}{q}.$$ By Kouba's interpolation theorem, we have $$X_n(r,r') = (X_n(s,s), X_n(p,q))_\theta.$$ Hence by the interpolation property of $\text{UMD}$ constant, we have $$ \beta_2(X_n(r,r')) \le \beta_2(X_n(s,s))^{1-\theta} \beta_2(X_n(p,q))^\theta.$$ By definition, we have $X_n(s,s) \simeq C_s,$  and hence $\beta_2(X_n(s,s)) = 1$. Combining with the first step, we have $$\lim_{n\to\infty} \beta_2(X_n(p,q)) \ge \lim_{n \to \infty} \beta_2(X_n(r,r'))^{1/\theta} = \infty.$$
\end{proof}

\section{Further results}
In this section, we give some equivalent conditions for $S_p[E]$ to be $\text{UMD}$, or equivalently, for $E$ to be $\text{OUMD}_p$. We give the equivalence between the $\text{UMD}$ property and the boundedness of the triangular projection on $S_p[E]$. Applying this equivalence, we prove that $E$ is $\text{OUMD}_p$ if and only if $E$ is $\text{OUMD}_p$ with respect to the so-called canonical filtration of matrix algebras.

We first give the following simple observation.
\begin{prop}
Let $1<p<\infty$, if we denote by $\mathcal{R}$ the Riesz projection $\mathcal{R}: L_p(\T, m) \rightarrow L_p(\T, m)$ defined by
\begin{displaymath}
\sum_{\text{finite}} x_n z^n \mapsto\sum_{n\geq 0} x_n z^n.
\end{displaymath}
Then $S_p[E]$ is $\text{UMD}$ if and only if
\begin{displaymath}
\mathcal{R}_E := Id_E \otimes \mathcal{R}: L_p(\T, m;E) \rightarrow L_p(\T, m;E)
\end{displaymath}
is completely bounded.
\end{prop}
\begin{proof}
By the classical results on $\text{UMD}$ property, $S_p[E]$ is $\text{UMD}$ if and only if the corresponding Riesz projection 
\begin{displaymath}
\mathcal{R}_{S_p[E]} : L_p(\T, m;S_p[E]) \rightarrow L_p(\T, m;S_p[E])
\end{displaymath}
is bounded. By the noncommutative Fubini theorem, the natural identification gives complete isometry
\begin{displaymath}
L_p(\T, m; S_p[E]) \simeq S_p[L_p(\T,m;E)].
\end{displaymath}
In this identification, $\mathcal{R}_{S_p[E]}$ corresponds to
\begin{displaymath}
Id_{S_p} \otimes \mathcal{R}_E: S_p[L_p(\T,m;E)] \rightarrow S_p[L_p(\T,m;E)].
\end{displaymath}
A very useful result in \cite{Pisier2} tell us that $\|\mathcal{R}_E\ncb = \|Id_{S_p} \otimes \mathcal{R}_E \|$. Hence we have
\begin{displaymath}
\|\mathcal{R}_E\ncb = \| \mathcal{R}_{S_p[E]} \|.
\end{displaymath}
This ends our proof. 
\end{proof}

The next theorem can be viewed as a special case of a result in \cite{Ricard_transfer}.
\begin{thm}
Let $T_E$ be the triangular projection on $S_p[E]$ defined by
\begin{displaymath}
(x_{ij}) \mapsto (x_{ij}1_{j\geq i}).
\end{displaymath}
Then $\|T_E\ncb = \|T_E\| = \|\mathcal{R}_E\ncb$.
\end{thm}

We refer to \cite{Junge_Xu_const} and \cite{Non_BR} for details on the canonical matrix filtration. As usual, we regard $M_n$ as a non-unital subalgebra of $M_\infty = B(\ell_2)$ by viewing an $n \times n$ matrix as an infinite one whose left upper corner of size $n \times n$ is the given $n\times n$ matrix, and all other entries are zero. The unit of $M_n$ is the projection $e_n \in M_\infty$ which projects a sequence in $\ell_2$ into its first $n$ coordinates. The canonical matrix filtration is the increasing filtration $(M_n)_{n\geq 1}$ of subalgebras of $M_\infty$. We denote by $E_n: M_\infty \rightarrow M_n$ the corresponding conditional expectation. It is clear that 
\begin{displaymath}
E_n(a) = e_n a e_n = \sum_{\max(i, j) \leq n} a_{ij}\otimes e_{ij}, \text{ for all } a = (a_{ij})\in M_\infty.
\end{displaymath}
\begin{rem}
Note that $E_n$ is not faithful, thus the noncommutative martingales with respect to the filtration $(M_n)_{n \geq 1}$ are different from the usual ones.  But this difference is not essential for what follows.
\end{rem}

We can define the $\text{OUMD}_p$ property with respect to this canonical matrix filtration. Let $ x \in S_p[E]$. Then \begin{displaymath}
d_1 x = E_1(x), \quad d_n x = E_n(x) - E_{n-1}(x), \text{ for all } n \geq 2.
\end{displaymath}
$E$ is said to be $\text{OUMD}_p$ with respect to the canonical matrix filtration, if there exists a constant $K$ depending only on $p$ and $E$, such that for all positive integers $N$ and all choices of signs $\varepsilon_n = \pm 1$, we have
\begin{displaymath}
\| \sum_{n=1}^N \varepsilon_n d_n x\|_{S_p[E]} \leq K \|x\|_{S_p[E]}. 
\end{displaymath}
Let $K_p(E)$ denote the best such constant.

Every choice of signs $\varepsilon$ generates a transformation $T_\varepsilon$ defined by
\begin{displaymath}
T_\varepsilon (x) = \sum_n \varepsilon_n d_n x.
\end{displaymath}
An element $x \in S_p[E]$ is said to have finite support if the support of $x$ defined by $\text{supp}(x) = \{ (i,j) \in \N^2: x_{ij} \neq 0 \}$ is finite. Note that  $T_\varepsilon$ is always well-defined on the subspace of finite supported elements.

An operator space $E$ is $\text{OUMD}_p$ with respect to the canonical matrix filtration if for every choice of signs $\varepsilon$, we have
\begin{displaymath}
\| T_\varepsilon (x) \|_{S_p[E]} \leq K_p(E) \|x\|_{S_p[E]}, \quad |\text{supp}(x)| < \infty.
\end{displaymath}

\begin{rem}
The transformation $T_\varepsilon$ is a Schur multiplication associated with the function $ f_\varepsilon (i, j) = \varepsilon_{\max(i,j)}$. Indeed, pick up an arbitrary element $x = (x_{ij}) \in S_p^N[E]$, we have 
\begin{displaymath}
d_n x = \sum_{\max(i, j) \leq n} x_{ij} \otimes e_{ij} - \sum_{\max(i, j) \leq n -1} x_{ij} \otimes e_{ij} = \sum_{\max(i, j) = n} x_{ij} \otimes e_{ij},
\end{displaymath}
thus 
\begin{displaymath}
T_\varepsilon (x) = \sum_{n = 1}^N \varepsilon_n d_n x = \sum_{n = 1}^N \varepsilon_n \sum_{\max(i, j) = n} x_{ij} \otimes e_{ij} = (\varepsilon_{\max(i,j)} x_{ij}). 
\end{displaymath}
\end{rem}. 

\begin{rem}\label{integral}
Let $D_\varepsilon = \text{diag}\{ \varepsilon_1, \cdots, \varepsilon_n, \cdots \}$. Then $T_\varepsilon(x)$ multiplied on the left by the scalar matrix $D_\varepsilon$, we get $D_\varepsilon T_\varepsilon(x) = (\varepsilon_i \varepsilon_{\max(i,j)} x_{ij})$. After taking the average according to independent uniformly distributed choices of signs, we get the lower triangular projection of $x$, i.e, we have  
\begin{displaymath}
\int D_\varepsilon T_\varepsilon (x) d \varepsilon = \int (\varepsilon_i \varepsilon_{\max(i,j)} x_{ij}) d \varepsilon = (x_{ij} 1_{i \geq j} ).
\end{displaymath}
\end{rem}

The following result is inspired by \cite{Junge_Xu_const} and \cite{Non_BR}
\begin{thm}
Let $1 < p < \infty$. Then $E$ is $\text{OUMD}_p$ if and only if it is $\text{OUMD}_p$ with respect to the canonical matrix filtration. Moreover, we have: 
\begin{displaymath}
\frac{1}{2} (K_p(E) - 1) \leq \|T_E\| \leq K_p(E).
\end{displaymath}
\end{thm}
\begin{proof}
Assume that $E$ is $\text{OUMD}_p$. Then $S_p[E]$ is $\text{UMD}$ and the triangular projection $T_E$ is bounded. Let $T_E^{-}$ be the triangular projection defined by $(x_{ij}) \mapsto (x_{ij}1_{j\leq i})$, it is clear that $\|T_E\| = \|T_E^{-}\|$. We have
\begin{displaymath}
d_n x = d_nT_E x + d_n T_E^{-} - D_n x, 
\end{displaymath}
where $D_n x = e_{nn} x e_{nn}$. Thus 
\begin{eqnarray}\nonumber
\| \sum \varepsilon_n d_n x \|_{S_p[E]} &\leq & \|\sum \varepsilon_n d_n T_E x \|_{S_p[E]} + \|\sum \varepsilon_n d_n T_E^{-} x \|_{S_p[E]} \\
& &+ \|\sum \varepsilon_n D_n x \|_{S_p[E]}\nonumber.
\end{eqnarray}
Since $d_n T_E x$ is the $n$-th column of $T_E x$, it is easy to see
\begin{displaymath}
\|\sum \varepsilon_n d_n T_E x \|_{S_p[E]} = \|\sum d_n T_E x \|_{S_p[E]} = \|T_E x \|_{S_p[E]} \leq \|T_E\| \|x\|_{S_p[E]}.
\end{displaymath}
The same reason shows that 
\begin{displaymath}
\|\sum \varepsilon_n d_n T_E^{-} x \|_{S_p[E]}=\|\sum d_n T_E^{-} x \|_{S_p[E]}=\|T_E^{-} x \|_{S_p[E]}\leq \|T_E^{-}\| \|x\|_{S_p[E]}.
\end{displaymath}
For the third term, we have obviously that 
\begin{displaymath}
\|\sum \varepsilon_n D_n x \|_{S_p[E]}= \|\sum D_n x \|_{S_p[E]}\leq \|x\|_{S_p[E]}.
\end{displaymath}
Combining these inequalities, we have 
\begin{displaymath}
\| \sum \varepsilon_n d_n x \|_{S_p[E]} \leq (\|T_E\|+\|T_E^{-}\| +1) \|x\|_{S_p[E]} =  (2\|T_E\|+1) \|x\|_{S_p[E]}.
\end{displaymath}
So $E$ is $\text{OUMD}_p$ with respect to the canonical matrix filtration with $K_p(E) \leq 2\|T_E\|+1$.

Conversely, assume that $E$ is $\text{OUMD}_p$ with respect to the canonical matrix filtration. We shall show that $E$ is $\text{OUMD}_p$. It suffices to show that the triangular projection $T_E$ is bounded. According to the remark \ref{integral}, we have
\begin{displaymath}
\| (x_{ij} 1_{i \geq j} )\|_{S_p[E]} \leq \int \| D_\varepsilon T_\varepsilon (x) \|_{S_p[E]} d \varepsilon \leq K_p(E) \| x \|_{S_p[E]},
\end{displaymath}
proving that $\| T_E^{-} \| \leq K_p(E)$, and hence $\| T_E \| \leq K_p(E)$.
\end{proof}

\begin{rem}
We have a slightly better estimation for $\| T_E \|$ as the following
\begin{displaymath}
\frac{1}{2}(K_p(E) - 1) \leq  \| T_E \| \leq \frac{1}{2}(K_p(E) + 1).
\end{displaymath}
We omit the proof here.
\end{rem}

\section*{Appendix}
In this appendix, we quickly review the results from \cite{Musat1} that are unaffected by the already  mentioned gap.

We first recall that it follows from \cite{Pisier_Xu_en} that for any $1< p< \infty$, $S_p$ has $\text{OUMD}_p$.
\begin{thm}[Musat]\label{easy_fact}
Let $1< p, q < \infty$. Then $$C_p \text{ has } \text{OUMD}_q.$$
\end{thm}
\begin{proof}
Let $$S = \Big \{ (\frac{1}{p},\frac{1}{q}) \in  (0, 1)\times (0, 1): C_p \text{ has } \text{OUMD}_q \Big\}.$$ We need to show that $$ S = (0, 1)\times (0, 1).$$  By complex interpolation, it is clear that $S$ is convex. Since $C_p$ and $C_{p'} = R_p$ are subspaces of $S_p$, by \cite{Pisier_Xu_en}, both of them have $\text{OUMD}_p$.  Hence $(\frac{1}{p}, \frac{1}{p}) \in S$ and $(\frac{1}{p'}, \frac{1}{p})\in S$ for all $1< p < \infty$.  The result now follows since $$(0, 1) \times (0, 1) = \text{conv}\Big(\{(\frac{1}{p}, \frac{1}{p}): 1 < p < \infty \} \cup\{(\frac{1}{p'}, \frac{1}{p}): 1 < p < \infty \} \Big).$$
\end{proof}

\begin{thm}[Musat]
Let $L$ be the set of interior points of the closed convex set $\text{conv}\Big\{ (0, 0), (1, 1), (\frac{1}{2}, 0), (\frac{1}{2}, 1) \Big\}$. Then $$S_p \text{ has } \text{OUMD}_q \text{ if } (\frac{1}{p}, \frac{1}{q}) \in L.$$
\end{thm}
\begin{proof}
We have $$L=  \text{conv}\Big(\{(\frac{1}{p}, \frac{1}{p}): 1 < p < \infty \} \cup\{(\frac{1}{2}, \frac{1}{q}): 1 < q < \infty \} \Big).$$
As operator spaces, we have completely isometrically $$S_2 \simeq OH \simeq C_2,$$  hence by Theorem \ref{easy_fact}, $S_2$ has $\text{OUMD}_q$ for all $1< q < \infty$. If we define $$T = \Big \{ (\frac{1}{p},\frac{1}{q}) \in  (0, 1)\times (0, 1): S_p \text{ has } \text{OUMD}_q \Big\},$$ then $T$ is convex and contains all the points $(\frac{1}{p}, \frac{1}{p}), 1< p< \infty$ and all the points $(\frac{1}{2}, \frac{1}{q}), 1< q < \infty$.  Hence $$L \subset T.$$
\end{proof}

\section*{Acknowledgements}
The author is extremely grateful to his advisor Gilles Pisier for helpful discussions. He also thanks Quanhua Xu for the helpful suggestion during the preparation of this paper. He would like thank Javier Parcet for explaining carefully the gap in \cite{Musat1} to him.


\def\cprime{$'$}

\end{document}